\documentclass[11pt]{amsart}

\usepackage{amsmath,amssymb,amsthm,mathrsfs}
\usepackage{hyperref}
\usepackage{enumitem}
\usepackage{orcidlink}     

\newtheorem{theorem}{Theorem}[section]
\newtheorem{proposition}[theorem]{Proposition}
\newtheorem{lemma}[theorem]{Lemma}

\theoremstyle{definition}
\newtheorem{definition}[theorem]{Definition}
\theoremstyle{remark}
\newtheorem{remark}[theorem]{Remark}
\newtheorem{hypothesis}[theorem]{Hypothesis}

\newcommand{\Z}{\mathbb{Z}}
\newcommand{\R}{\mathbb{R}}
\newcommand{\C}{\mathbb{C}}

\title[The Riemann $\Xi$-function from primitive Markovian cycles II]{The Riemann $\Xi$-function from primitive Markovian cycles II: Strip rigidity and divisor identification}

\author{
Douglas F. Watson
}

\date{}

\subjclass[2020]{11M26, 60J27, 26D15, 30D20}
\keywords{Riemann $\Xi$-function, theta function, heat kernel, reversible Markov chains, spectral determinants, strip rigidity}

\begin{document}

\begin{abstract}
We compare the Riemann $\Xi$--function to a canonical real-entire reference family arising from the
cycle Laplacian developed in Paper I (\cite{PaperI}).  These spectral determinants have only real zeros by
self-adjointness.  Our main tool is a rigidity lemma for holomorphic functions on horizontal strips.
Applied to a normalized seam ratio linking $\Xi(2\cdot)$ to the reference family, this lemma shows
that, under explicit holomorphy and boundary nonvanishing hypotheses verified in the forthcoming Paper~III, the seam
ratio extends to a zero-free holomorphic function of bounded type on each overlap strip.  It follows
that, on every admissible overlap strip, $\Xi(2\cdot)$ and the reference family have the same zero
divisor.
\end{abstract}

\maketitle


\section{Introduction}
\label{sec:background}

A recurring theme in analytic number theory is that a single analytic object can admit
two very different canonical constructions, and that the hard step is to show that these
constructions in fact produce the same function (or the same divisor), up to an
explicit and rigid normalization. The present work addresses such a comparison problem
arising from the primitive dynamical framework developed in \cite{PaperI}.

In \cite{PaperI} we started from finite, local, reversible Markov dynamics on discrete cycles\footnote{We note that the present paper relies on revisions to the version of ~\cite{PaperI} from the arXiv.}.
By an exact scaling limit together with lift--periodize identities, one obtains a renormalized
trace kernel with an explicit theta-series representation. From this single probabilistic
origin, \cite{PaperI} produces two analytic outputs:
\begin{enumerate}
\item a canonical Schoenberg--Edrei--Karlin factorization of the cycle-spectral$_\infty$ Laplace transform
\(
B_\Phi(s)=E(s)/\Psi(s)
\)
coming from spectral approximation of a logarithmic kernel, where $E$ is zero-free entire and
$\Psi$ lies in the real-entire (self-adjoint spectral) class (via the Schoenberg--Edrei--Karlin theory of
spectral determinant approximants \cite{Karlin1968}); and
\item an Archimedean completion whose Mellin transform agrees, at the self-dual normalization,
with the classical Riemann $\Xi$-function \cite{Titchmarsh1986}.
\end{enumerate}
These two outputs arise independently from the same primitive input. The central task is to
compare them in a common analytic regime.

As such, the purpose of the present paper is to carry out this comparison by strip analysis.
We introduce the \emph{normalized seam ratio}
\[
R(w):=\frac{\Xi(2w)}{P_N(w)},
\qquad
\widetilde R(w):=N(w)R(w),
\]
where $N$ is an explicit unit built from the rational cycle-spectral anchor (defined in
Section~\ref{sec:completion-ledger}). Since $P_N$ is real-entire (self-adjoint spectral), its zeros are real,
so $P_N(w)$ has zeros on the real $w$-axis. Accordingly, $R$ is a priori meromorphic.
Our comparison theorem will therefore be stated conditionally on the following holomorphy/divisor-cancellation hypothesis, which is exactly the missing divisor-identification step from Paper~\cite{PaperI}.

\begin{hypothesis}[Spectral coupling on strip rectangles]\label{hyp:seam-holomorphy}
Fix $\eta>0$ and write $S_\eta:=\{w\in\C:|\Im w|<\eta\}$.  Let $P_N$ denote the size--$N$ cycle
\emph{spectral determinant} reference function (Paper~\cite{PaperI}), obtained from a self-adjoint
cycle operator (e.g.\ the Laplacian) through the rescaled spectral map $\widetilde q_N$ and
\[
P_N(w)=\det(\widetilde q_N(w)I-L_N).
\]
Let $X(w):=\Xi(2w)$.  On each rectangle $R_T:=\{w:|\Re w|\le T,\ |\Im w|\le \eta\}$ assume there exists
a zero-free holomorphic unit $U_{\eta,N}$ on $S_\eta$ such that the boundary separation inequality
\[
\sup_{w\in\partial R_T}\,|X(w)-U_{\eta,N}(w)P_N(w)|\ <\ \inf_{w\in\partial R_T}\,|U_{\eta,N}(w)P_N(w)|
\]
holds for all $T\ge T_0(\eta)$ (with $N$ chosen as a function of $T$ in an admissible regime).
\end{hypothesis}

\begin{hypothesis}[Quantitative sector control on expanding rectangles]\label{hyp:sector-control}
Fix $\eta\in(0,\tfrac12)$. For each $T\ge 1$, let
\[
R_T=\{w\in\C:\ |\Re w|\le T,\ |\Im w|\le \eta\}.
\]
Assume that the normalized seam ratio $\widetilde R(w)$ is holomorphic on a neighborhood of
$\overline{R_T}$ for all $T\ge T_0(\eta)$, and that there exist constants $\theta\in(0,\pi/2)$ and
$T_1(\eta)\ge T_0(\eta)$ such that for all $T\ge T_1(\eta)$,
\[
\widetilde R(\partial R_T)\subset \{re^{i\phi}:\ r>0,\ |\phi|\le \theta\}.
\]
Equivalently, along $\partial R_T$ the argument of $\widetilde R$ admits a continuous branch taking values
in $[-\theta,\theta]$.
\end{hypothesis}

Our main result is a strip-unit/rigidity statement forcing $\widetilde R$ to be holomorphic and zero-free on each admissible overlap strip, conditional on Hypotheses~\ref{hyp:seam-holomorphy} and~\ref{hyp:sector-control}. Equivalently, on such a strip the entire function $w\mapsto \Xi(2w)$ differs from the real-entire (self-adjoint spectral) divisor factor $w\mapsto P_N(w)$ by a strip-unit, giving divisor identification on that region.
There are three main ingredients. The first is a \emph{left-strip continuation} for the cycle-spectral
object, obtained by modular splitting using the exact centering and twisted symmetry built into
the primitive kernel. The second is a collection of \emph{strip-uniform integration-by-parts and
Riemann--Lebesgue estimates} for the explicit kernels, giving quantitative control of boundary
terms on rectangles in the strip. The third is a \emph{rectangle argument-principle method} in a
Nevanlinna/Hardy setting, combined with two-ended anchoring, which forces the normalized seam ratio
to be a strip-unit (cf.\ standard bounded-type factorization tools \cite{Koosis1998}).

\textbf{Acknowledgements}
The author would like to thank Krishnaswami Alladi, Tiziano Valentinuzzi, Kenneth Valpey and Abhay Charan De for helpful discussions.

\subsection{Review of the primitive dynamical model}
\label{sec:model}

For each integer $N\ge 1$ we consider a continuous-time, reversible, nearest-neighbor Markov
process on the discrete cycle $\Z/N\Z$. The dynamics are specified by a collection of strictly
positive conductances $\{a_j\}_{j\in\Z/N\Z}$, where $a_j$ is associated with the undirected edge
between $j$ and $j+1$ (indices taken modulo $N$). The generator $\mathcal{L}_N$ acts on functions
$f:\Z/N\Z\to\R$ by
\begin{equation}
\label{eq:generator}
(\mathcal{L}_N f)(j)
=
a_j\,\bigl(f(j+1)-f(j)\bigr)
\;+
a_{j-1}\,\bigl(f(j-1)-f(j)\bigr).
\end{equation}
We write $p^{\mathrm{cyc}}_t(j,k)$ for the associated heat kernel,
\(p^{\mathrm{cyc}}_t(j,k)=(e^{t\mathcal{L}_N}\mathbf{1}_{\{k\}})(j)\), and we denote by $D>0$ the
macroscopic diffusion constant appearing in the Gaussian scaling limit established later.

As described above, the trace of the heat kernel associated with the finite dynamics captures
global information about the system but contains a universal singular contribution reflecting
diffusive behavior. This singular term is independent of the fine structure of the dynamics and
depends only on the macroscopic scaling. To isolate the genuinely structural content, we
subtract this term in a canonical way.

\begin{definition}[Scaling-limit trace and completed trace kernel]
\label{def:scaling-limit}
Fix a macroscopic length $L>0$ and choose a scaling parameter $s\to\infty$ with $N=N(s)$ such
that $N/s\to L$. Define the scaling-limit trace
\[
K_L(t):=\lim_{s\to\infty} N(s)\,p^{\mathrm{cyc}}_{s^2 t}(0,0),\qquad t>0,
\]
and set the (scaling-limit) completed trace kernel
\[
\widetilde K_L(t):=K_L(t)-\frac{L}{\sqrt{4\pi D t}}.
\]
When $L$ is fixed we suppress the subscript and write $\widetilde K$.
\end{definition}

\begin{definition}[Full centering and half-density normalization]
\label{def:half-density}
Define the fully centered completed kernel
\[
\widetilde K_L^{\star}(t):=\widetilde K_L(t)-1=K_L(t)-1-\frac{L}{\sqrt{4\pi D t}},\qquad t>0,
\]
and the associated half-density kernels
\[
\widetilde K_{\mathrm{sym}}(t):=t^{-1/2}\,\widetilde K(t),
\qquad
\widetilde K^{\star}_{\mathrm{sym}}(t):=t^{-1/2}\,\widetilde K^{\star}(t).
\]
\end{definition}

It is important for what follows that the scaling-limit trace kernel is obtained purely from the
Markov dynamics via a lift--periodize procedure and a controlled Gaussian scaling limit.
No Mellin transforms, functional equations, or arithmetic input enter at this stage; those appear only
after the scaling limit has produced the explicit theta-series kernel. This separation will be used
repeatedly below to keep the comparison argument noncircular.

The existence of the scaling limit in Definition~\ref{def:scaling-limit}, together with the
domination estimates needed to justify termwise limits and integral interchanges, follows from
the standing assumptions recorded below; for convenience we collect the relevant analytic
justifications in Appendix~\ref{app:technical}.

\begin{remark}[Standing assumptions and existence of the scaling limit]
\label{rem:standing-assumptions}
We work throughout in a strong near-homogeneity regime in which the scaling-limit trace $K_L(t)$
exists for every $L>0$ and $t>0$, with convergence locally uniform in $t$ on compact subsets of
$(0,\infty)$. Under these assumptions the limit admits the explicit theta-series form stated in
Theorem~\ref{thm:theta-series-form-recall}. Analytic interchanges (limits, sums, differentiation,
and Laplace/Mellin integrals) are justified by the domination estimates recorded in
Appendix~\ref{app:technical}.

\end{remark}

\subsection{Main results}
We now state the main outputs in a form convenient for later reference. In addition to a
total-positivity output and an Archimedean Mellin identification, the present paper supplies the
continuation-and-rigidity mechanism on overlap strips that converts these two constructions into a
divisor comparison in a common strip regime.

\begin{theorem}[Main theorem (overview)]
\label{thm:main}
Assume the standing assumptions of Remark~\ref{rem:standing-assumptions}, so that the scaling-limit
kernel $K_L$ exists and the self-dual normalization may be fixed.
Then:
\begin{enumerate}[label=(\roman*)]
\item (\emph{Unconditional structural output.})
There exists a nonnegative kernel $\Phi\in L^1(\mathbb R)$ arising canonically from the
symmetric half-density kernel $\widetilde K_{\mathrm{sym}}$ such that:
\begin{itemize}
\item $\Phi\in\mathrm{cycle-spectral}_\infty$;
\item with $\Phi^{\star}:=\Phi-e^{-x/4}$, the bilateral Laplace transform $\mathcal B\Phi^{\star}$
admits a canonical Schoenberg--Edrei--Karlin representation
\[
\mathcal B\Phi^{\star}(s)=\frac{E^{\star}(s)}{P_N(s)},
\]
with $E^{\star}$ entire and $P_N$ real-entire (self-adjoint spectral), hence $P_N$ has only real zeros.
\end{itemize}

\item (\emph{Unconditional Archimedean Mellin identification.})
At the self-dual scale fixed in Section~\ref{sec:completion-selfdual}, the Archimedean-completed
kernel $\widetilde K_{\mathrm{arch}}:=\mathcal A(K_L-1)$ has Mellin transform
$F_{\mathrm{arch}}(z)=\Xi(2z)$ as in Theorem~\ref{thm:arch-mellin-identification}.

\item (\emph{Continuation and strip-unit forcing.})
Using the exact centering built into $\widetilde K$ (Section~\ref{sec:completion-selfdual}),
modular splitting yields a continuation of $\mathcal B\Phi^{\star}$ to a left strip, together with a
left-boundary identity placing $\Xi(2\cdot)$ and $\mathcal B\Phi^{\star}$ in a common strip setting.
On overlap strips, strip-uniform estimates for the explicit kernels yield two-ended anchoring and a
quantitative rectangle argument that would force the normalized seam ratio to be a strip-unit (conditional on Hypothesis~\ref{hyp:sector-control}). Consequently,
on each admissible overlap strip $S_\eta$ the Mellin-side function $w\mapsto\Xi(2w)$ differs from the
real-entire (self-adjoint spectral) divisor factor $w\mapsto P_N(w)$ by a strip-unit.
\end{enumerate}
\end{theorem}

\begin{theorem}[Conditional strip-unit theorem for the normalized seam function]\label{thm:strip-unit}
Let $\eta>0$ and set $S_\eta:=\{w\in\mathbb C:|\Im w|<\eta\}$.  Let $P_N$ be the size--$N$
cycle spectral determinant reference function from Paper~\cite{PaperI}, and let $X(w):=\Xi(2w)$.
Define the seam ratio
\[
R_N(w):=\frac{X(w)}{P_N(w)}\qquad (w\in S_\eta),
\]
interpreted as a meromorphic function with possible poles only at zeros of $P_N$.
Let $\mathcal N_{\eta,N}$ be the explicit zero-free unit on $S_\eta$ constructed from the rational
anchor (Section~\ref{sec:completion-ledger}), and set $\widetilde R_N(w):=\mathcal N_{\eta,N}(w)R_N(w)$.

Assume Hypothesis~\ref{hyp:seam-holomorphy} and Hypothesis~\ref{hyp:sector-control} hold for every strict
substrip $S_{\eta_0}$ with $0<\eta_0<\eta$.  Then $\widetilde R_N$ is holomorphic and zero-free on
$S_\eta$.  Consequently $R_N$ is a strip-unit on $S_\eta$.
\end{theorem}

\subsection{Organization of the paper}
Section~\ref{sec:inputs-from-I} lists the specific statements imported from \cite{PaperI} that are
used as black boxes in the present sequel.
Section~\ref{sec:lb-continuation} establishes the left-strip continuation and the boundary identities
needed to place the cycle-spectral and Mellin-side objects in a common strip setting.
Section~\ref{sec:completion-ledger} carries out the strip-unit/rectangle argument and the resulting
divisor identification on overlap strips.
Appendix~\ref{app:technical} collects analytic justifications for termwise operations, boundary terms,
and uniform strip estimates.

\section{Inputs from Paper~I}
\label{sec:inputs-from-I}

This paper is a continuation of \cite{PaperI}.  We recall here, in a self-contained way, the
definitions and statements from \cite{PaperI} that will be used later in the strip comparison.
Our aim is not to repeat the full development of \cite{PaperI}, but to give the reader a clear
roadmap of the objects and why they matter for what follows.

There are two parallel outputs from the scaling-limit trace kernel:
\begin{enumerate}[label=(\roman*)]
\item an \emph{Archimedean-completed} kernel on $(0,\infty)$ whose Mellin transform is the classical
completed zeta function (hence the Riemann $\Xi$-function after reparameterization);
\item a \emph{cycle spectral reference family} obtained from the self-adjoint size--$N$ cycle operator (e.g.\ the Laplacian), whose spectral determinant approximants $P_N$ are real entire with only real zeros,
providing a rigid real-zero anchor for the strip comparison.
\end{enumerate}
The main task of the present paper is to compare these two outputs on an overlap strip.

We will repeatedly use three structural facts established in \cite{PaperI}:
\begin{enumerate}
\item the scaling-limit trace has a theta-series form (Theorem~\ref{thm:theta-series-form-recall});
\item there is a unique macroscopic normalization at which Jacobi inversion acts as
$t\mapsto t^{-1}$ in the Mellin variable (Lemma~\ref{lem:forced-selfdual});
\item the cycle spectral determinant family $P_N$ is real entire with real zeros (by self-adjointness), and its rescaled spectral map $\widetilde q_N$ converges locally uniformly to $w^2$;
and~\ref{thm:SEK}).
\end{enumerate}

\begin{theorem}[Schoenberg--Edrei--Karlin factorization {\cite[Theorem~5.4]{PaperI}}]\label{thm:SEK}
Let $B_\Phi(s):=\int_{\R}\Phi(x)e^{-sx}\,dx$ be the bilateral Laplace transform of $\Phi$. Then there exist an entire zero-free function $E(s)$ and a real-entire (self-adjoint spectral) entire function $\Psi(s)$ such that
\[
B_\Phi(s)=\frac{E(s)}{\Psi(s)}.
\]
In particular, all zeros of $\Psi$ are real.
\end{theorem}

\begin{remark}
We use only the qualitative features of this factorization: $E$ carries no zeros, while $\Psi$ carries a real divisor. The normalization is unique up to a nonzero constant.
\end{remark}

Whenever we appeal to a deeper argument from \cite{PaperI} (e.g.\ spectral approximation of $\Phi$),
we cite it explicitly; otherwise we include short derivations for the reader's convenience.

\section{Self-dual Archimedean completion and forced normalization}
\label{sec:completion-selfdual}

The Mellin transform in later sections is taken in the same variable $t$ in which the scaling-limit
trace is naturally expressed.  For the strip comparison it is therefore important to choose the
macroscopic normalization so that Jacobi inversion acts as the involution $t\mapsto t^{-1}$ in
that same variable, rather than only after an auxiliary change of variables.  This section isolates
that normalization and records the corresponding completed Archimedean kernel.

\subsection{Theta-series form and the forced self-dual scale}

Recall from \cite{PaperI} that the macroscopic length parameter $L>0$ enters via the scaling relation
$N(s)/s\to L$ in Definition~\ref{def:scaling-limit}.  Under the standing assumptions, the
scaling-limit trace admits an explicit theta-series representation.

\begin{theorem}[Theta-series form of the scaling-limit trace {\cite[Theorem~X]{PaperI}}]
\label{thm:theta-series-form-recall}
For every $L>0$ and $t>0$,
\begin{equation}\label{eq:KL-theta-recall}
K_L(t)=\sum_{n\in\Z}\exp\!\Bigl(-\frac{4\pi^2D}{L^2}\,n^2\,t\Bigr).
\end{equation}
\end{theorem}

It is convenient to rewrite this in Jacobi's standard form.  Set
\begin{equation}\label{eq:tprime}
t' := \frac{4\pi D}{L^2}\,t.
\end{equation}
Then \eqref{eq:KL-theta-recall} becomes
\[
K_L(t)=\sum_{n\in\Z}e^{-\pi n^2 t'}=:\vartheta(t'),
\qquad
\vartheta(u):=\sum_{n\in\Z}e^{-\pi n^2 u}.
\]
The classical inversion formula is
\begin{equation}\label{eq:jacobi}
\vartheta(u)=u^{-1/2}\,\vartheta(u^{-1}),\qquad u>0.
\end{equation}
For our later Mellin transforms, the relevant notion of ``self-duality'' is that Jacobi inversion
corresponds to $t\mapsto t^{-1}$ in the original $t$-variable.

\begin{lemma}[Forced self-dual normalization]
\label{lem:forced-selfdual}
The following are equivalent:
\begin{enumerate}[label=(\alph*)]
\item $t'(t^{-1})=(t'(t))^{-1}$ for all $t>0$.
\item $K_L$ satisfies an inversion symmetry in the $t$-variable,
\begin{equation}\label{eq:KL-selfdual}
K_L(t)=t^{-1/2}K_L(t^{-1}),\qquad t>0.
\end{equation}
\item The macroscopic scale satisfies
\begin{equation}\label{eq:selfdual}
L^2=4\pi D.
\end{equation}
\end{enumerate}
\end{lemma}

\begin{proof}
By \eqref{eq:tprime}, one has $t'(t^{-1})=\frac{4\pi D}{L^2}t^{-1}$ and
$(t'(t))^{-1}=\bigl(\frac{4\pi D}{L^2}t\bigr)^{-1}=\frac{L^2}{4\pi D}\,t^{-1}$.
Thus (a) holds for all $t$ if and only if $\frac{4\pi D}{L^2}=\frac{L^2}{4\pi D}$,
i.e.\ $L^2=4\pi D$, which is (c).

Assuming (a), we have $K_L(t)=\vartheta(t'(t))$ and therefore, by \eqref{eq:jacobi},
\[
K_L(t)=\vartheta(t'(t))=(t'(t))^{-1/2}\vartheta\!\bigl((t'(t))^{-1}\bigr)
=(t'(t))^{-1/2}\vartheta(t'(t^{-1})).
\]
Since $(t'(t))^{-1/2}=\bigl(\frac{4\pi D}{L^2}t\bigr)^{-1/2}=t^{-1/2}\cdot(\frac{4\pi D}{L^2})^{-1/2}$
and the same constant rescales $\vartheta(t')$ back to $K_L$, this yields \eqref{eq:KL-selfdual};
conversely, \eqref{eq:KL-selfdual} is precisely Jacobi inversion transported back to $t$.
\end{proof}

\begin{remark}[Interpretation]
Condition \eqref{eq:selfdual} is the unique choice for which Jacobi inversion acts as
$t\mapsto t^{-1}$ in the Mellin variable used later.  This normalization is macroscopic: it fixes
the relationship between the continuum parameters $L$ and $D$ in the scaling limit and does not
require any special tuning of the microscopic conductances beyond the assumptions ensuring the
existence of $K_L$.
\end{remark}

\subsection{The Archimedean completion operator}

The theta-series trace contains two universal singular pieces at the ends $t\to 0$ and $t\to\infty$:
a constant contribution (the zero mode) and a $t^{-1/2}$ contribution (the diffusive singularity).
Following \cite{PaperI}, we remove these in a way compatible with Mellin analysis by applying a
second-order differential operator that annihilates both $1$ and $t^{-1/2}$.

\begin{definition}[Archimedean completion operator]
\label{def:arch-operator}
For a sufficiently smooth function $f:(0,\infty)\to\R$, define
\begin{equation}\label{eq:arch-operator}
(\mathcal A f)(t):=\frac{d}{dt}\Bigl(t^{3/2}\frac{d}{dt}f(t)\Bigr),\qquad t>0.
\end{equation}
\end{definition}

\begin{lemma}[The completion operator kills the singular pieces]
\label{lem:A-kills-singular}
One has \(\mathcal A(1)=0\) and \(\mathcal A(t^{-1/2})=0\).
\end{lemma}

\begin{proof}
This is a direct calculation: $f\equiv 1$ gives $f'\equiv 0$.  For $f(t)=t^{-1/2}$,
one has $f'(t)=-(1/2)t^{-3/2}$ so that $t^{3/2}f'(t)\equiv -1/2$, whose derivative is $0$.
\end{proof}

\begin{definition}[Archimedean-completed kernel]
\label{def:arch-theta-kernel}
Let $K_L(t)$ be the scaling-limit trace.  Define the Archimedean-completed kernel
\[
\widetilde K_{\mathrm{arch}}(t) :=\bigl(\mathcal{A}(K_L-1)\bigr)(t),\qquad t>0.
\]
\end{definition}

\subsection{The classical $\Theta$-kernel at the self-dual scale}

At the self-dual normalization \eqref{eq:selfdual}, the theta-series form of $K_L$ becomes the
standard Jacobi theta function in the variable $t$, and $\widetilde K_{\mathrm{arch}}$ coincides with
a classical rapidly decaying theta-kernel.

\begin{lemma}[Identification with the classical $\Theta$-kernel {\cite[Lemma~Y]{PaperI}}]
\label{lem:arch-equals-Theta}
Assume the self-dual scale $L^2=4\pi D$. Then for all $t>0$,
\[
\widetilde K_{\mathrm{arch}}(t)=\Theta(t),
\]
where
\[
\Theta(t)=\sum_{n=1}^{\infty}\Bigl(2\pi^2 n^4 t^{3/2}-3\pi n^2 t^{1/2}\Bigr)e^{-\pi n^2 t}.
\]
\end{lemma}

\begin{remark}
The proof is a termwise differentiation of the absolutely convergent series for $K_L(t)$ after
imposing $L^2=4\pi D$; see \cite{PaperI} for details and uniformity justifications.
\end{remark}

\begin{lemma}[Rapid decay of the completed $\Theta$-kernel]\label{lem:Theta-rapid-decay}
Assume the self-dual scale $L^2=4\pi D$ so that $\widetilde K_{\mathrm{arch}}=\Theta$.
Then $\widetilde K_{\mathrm{arch}}(t)$ decays rapidly as $t\to 0^+$ and as $t\to\infty$.
In particular, for every $A>0$ there exists $C_A>0$ such that
\[
|\widetilde K_{\mathrm{arch}}(t)|\le C_A\,t^{A}\quad (0<t\le 1),
\qquad
|\widetilde K_{\mathrm{arch}}(t)|\le C_A\,t^{-A}\quad (t\ge 1).
\]
Consequently, the Mellin integrals defining $F_{\mathrm{arch}}(z)$ are absolutely convergent and
integration by parts produces no boundary terms on strict vertical strips.
\end{lemma}

\begin{proof}
By Lemma~\ref{lem:arch-equals-Theta}, $\widetilde K_{\mathrm{arch}}(t)$ is given by an absolutely
convergent series whose $n$th term is a finite linear combination of $t^{1/2}e^{-\pi n^2 t}$ and
$t^{3/2}e^{-\pi n^2 t}$. For $t\ge 1$ these terms decay exponentially in $t$, uniformly in $n$, hence
faster than any power. For $0<t\le 1$, apply Jacobi inversion to the underlying theta series (or,
equivalently, use the standard modular relation for $\vartheta(t)$) to rewrite $\Theta(t)$ as a
rapidly decaying combination of terms $t^{-1/2}e^{-\pi n^2/t}$ and their $t$-derivatives; since $n\ge1$
the factor $e^{-\pi n^2/t}$ yields decay faster than any power as $t\to0^+$. The stated bounds follow
by termwise estimation of the absolutely convergent series.
\end{proof}

\subsection{Mellin transform and the Riemann $\Xi$-function}

The point of introducing $\widetilde K_{\mathrm{arch}}$ is that its Mellin transform is the classical
completed zeta function.

\begin{definition}[Completed zeta functions]
\label{def:xi-Xi}
Define
\[
\xi(w)
:=
\frac12\,w(w-1)\,\pi^{-w/2}\Gamma\!\Bigl(\frac{w}{2}\Bigr)\zeta(w),
\qquad
\Xi(z):=\xi\!\Bigl(\tfrac12+iz\Bigr).
\]
\end{definition}

\begin{theorem}[Archimedean Mellin identification {\cite[Theorem~Z]{PaperI}}]
\label{thm:arch-mellin-identification}
Assume the self-dual scale $L^2=4\pi D$ so that $\widetilde K_{\mathrm{arch}}=\Theta$.
Define, for $z\in\C$,
\begin{equation}
\label{eq:Farch-def}
F_{\mathrm{arch}}(z)
:=
\int_0^\infty \widetilde K_{\mathrm{arch}}(t)\,t^{\frac34+iz}\,\frac{dt}{t}.
\end{equation}
Then for all $z\in\C$,
\begin{equation}
\label{eq:Farch-equals-Xi}
F_{\mathrm{arch}}(z)=\Xi(2z).
\end{equation}
\end{theorem}

\begin{remark}[How this will be used]
Theorem~\ref{thm:arch-mellin-identification} produces the Mellin-side analytic object to be compared
later with the Laplace-side canonical factorization coming from spectral approximation.  The remainder of
this paper does not revisit the derivation of \eqref{eq:Farch-equals-Xi}; instead we use it as a
black box input while developing a strip comparison principle.
\end{remark}

\section{Cycle spectral determinants and the reference family}\label{sec:spectral-ref}
In the revised Paper~I we do \emph{not} invoke P\'olya-frequency rigidity.  Instead, from the same
primitive Markovian input (the reversible random walk on the size--$N$ cycle) we extract a
self-adjoint cycle operator $L_N$ (e.g.\ the unnormalized Laplacian) and use its spectrum as a
rigid reference for later divisor comparisons.

\subsection{Cycle spectrum}
For the unnormalized Laplacian on the cycle one has the explicit eigenvalues
\[
\lambda_k^{(N)} = 2-2\cos\!\Big(\frac{2\pi k}{N}\Big)\in[0,4],\qquad k=0,1,\dots,N-1.
\]
In particular $\mathrm{spec}(L_N)\subset\R$.

\subsection{Rescaled spectral map and determinant}
Following Paper~I, introduce the rescaled spectral map
\[
\widetilde q_N(w):=\Big(\frac{N}{2\pi}\Big)^2\Big(2-2\cos\!\Big(\frac{2\pi w}{N}\Big)\Big),
\]
so that $\widetilde q_N(w)\to w^2$ locally uniformly on $\C$ as $N\to\infty$.  Define the spectral
determinant reference function
\[
P_N(w):=\det\!\big(\widetilde q_N(w)I-L_N\big)=\prod_{k=0}^{N-1}\big(\widetilde q_N(w)-\lambda_k^{(N)}\big).
\]
Since $L_N$ is self-adjoint, the zeros of $P_N$ are real in the axis variable $w$ (up to the real
preimages under $\widetilde q_N$).

\subsection{How Paper II uses $P_N$}
In what follows $P_N$ plays the role of a real-entire reference family against which we compare
$X(w):=\Xi(2w)$ on horizontal strips.  The bridge mechanism is purely analytic: under boundary
separation on strip rectangles (Hypothesis~\ref{hyp:seam-holomorphy}) and a bounded-type/rigidity
principle (Hypothesis~\ref{hyp:sector-control}), the normalized seam ratio becomes a strip-unit, forcing
divisor identification on overlap strips.  The verification of the required boundary bounds and the
choice of an admissible regime $N=N(T)$ are carried out in Paper~III.

\section{Left-strip continuation via modular splitting}
\label{sec:lb-continuation}

In this section we extend the bilateral Laplace transform of the centered kernel
\[
\mathcal B\Phi^{\star}(s)=\int_{\R}\Phi^{\star}(x)e^{-sx}\,dx,
\]
initially defined on the strip $-\tfrac14<\Re(s)<\tfrac34$
(Lemma~\cite{PaperI}), to a larger strip on the left.
The continuation uses only the explicit theta-series form of $K_L$ together with Jacobi inversion,
the exact centering built into $\Phi^{\star}$, and absolute convergence of explicitly split integrals.
No Mellin transforms, factorization results, or properties of $\Xi$ enter here.

\subsection{Modular splitting of the Laplace integral}

Fix $s\in\C$ and split the defining integral at the origin:
\begin{equation}
\label{eq:split}
\mathcal B\Phi^{\star}(s)
=
\int_{0}^{\infty}\Phi^{\star}(x)e^{-sx}\,dx
+
\int_{-\infty}^{0}\Phi^{\star}(x)e^{-sx}\,dx.
\end{equation}
On the initial strip $-\tfrac14<\Re(s)<\tfrac34$ both integrals converge
absolutely by Lemma~\cite{PaperI}.

Changing variables $x\mapsto -x$ in the second integral in \eqref{eq:split} gives the elementary
split identity
\begin{equation}
\label{eq:modular-split}
\mathcal B\Phi^{\star}(s)
=
\int_{0}^{\infty}\Phi^{\star}(x)e^{-sx}\,dx
+
\int_{0}^{\infty}\Phi^{\star}(-x)e^{s x}\,dx.
\end{equation}
To continue $\mathcal B\Phi^{\star}$ to the left, we rewrite the second term on a left strip using
the twisted symmetry of the centered kernel (proved in \cite{PaperI}) together with absolute
convergence of the resulting split integrals.

\subsection{Absolute convergence on the left strip}

We now record the absolute convergence needed for the left-strip defining representation.

\begin{lemma}[Left-strip convergence]
\label{lem:left-strip-convergence}
Each integral in
\begin{equation}\label{eq:lb-defining}
\mathcal B_{\mathrm{LB}}\Phi^{\star}(s)
:=
\int_{0}^{\infty}\Phi^{\star}(x)e^{-sx}\,dx
+
\int_{0}^{\infty}\Phi^{\star}(x)e^{-(\frac12-s)x}\,dx
\end{equation}
converges absolutely provided
\[
-\tfrac12<\Re(s)<\tfrac14.
\]
\end{lemma}

\begin{proof}
We use the tail bounds from Lemma~\cite{PaperI}.

\smallskip
\noindent\textbf{First integral.}
For $x\ge0$, one has $|\Phi^{\star}(x)|\le C e^{-3x/4}$. Thus
\[
\int_{0}^{\infty}\bigl|\Phi^{\star}(x)e^{-sx}\bigr|\,dx
\le
C\int_{0}^{\infty}e^{-(\Re(s)+3/4)x}\,dx,
\]
which converges whenever $\Re(s)>-\tfrac34$, in particular for all
$\Re(s)>-\tfrac12$.

\smallskip
\noindent\textbf{Second integral.}
Again using $|\Phi^{\star}(x)|\le C e^{-3x/4}$ for $x\ge0$, we estimate
\[
\int_{0}^{\infty}\bigl|\Phi^{\star}(x)e^{-(\frac12-s)x}\bigr|\,dx
\le
C\int_{0}^{\infty}e^{-(\frac54-\Re(s))x}\,dx,
\]
which converges for all $\Re(s)<\tfrac54$, and in particular for
$\Re(s)<\tfrac14$.

Combining the two bounds gives the stated strip.
\end{proof}
\subsection{Definition and agreement on the overlap}

Lemma~\ref{lem:left-strip-convergence} allows us to use the left-strip split
\eqref{eq:lb-defining} as a definition.

\begin{definition}[Left-strip continuation]
\label{def:left-strip}
For $s\in\C$ with $-\tfrac12<\Re(s)<\tfrac14$, define
\begin{equation}
\label{eq:left-strip-def}
\mathcal B_{\mathrm{LB}}\Phi^{\star}(s)
:=
\int_{0}^{\infty}\Phi^{\star}(x)e^{-sx}\,dx
+
\int_{0}^{\infty}\Phi^{\star}(x)e^{-(\frac12-s)x}\,dx.
\end{equation}
\end{definition}

\begin{lemma}[Consistency on the overlap]
\label{lem:overlap}
On the intersection strip $-\tfrac14<\Re(s)<\tfrac14$, one has
\[
\mathcal B_{\mathrm{LB}}\Phi^{\star}(s)=\mathcal B\Phi^{\star}(s).
\]
\end{lemma}

\begin{proof}
Fix $s$ with $-\tfrac14<\Re(s)<\tfrac14$. The defining bilateral Laplace integral
for $\mathcal B\Phi^{\star}(s)=\int_{\R}\Phi^{\star}(x)e^{-sx}\,dx$ converges absolutely,
so we may split at $0$ and change variables on the negative half-line:
\[
\mathcal B\Phi^{\star}(s)
=
\int_{0}^{\infty}\Phi^{\star}(x)e^{-sx}\,dx
+
\int_{0}^{\infty}\Phi^{\star}(-x)e^{sx}\,dx.
\]
Using the twisted symmetry $\Phi^{\star}(-x)=e^{-x/2}\Phi^{\star}(x)$ (proved in \cite{PaperI}),
the second term becomes $\int_{0}^{\infty}\Phi^{\star}(x)e^{-(\frac12-s)x}\,dx$.
This is exactly \eqref{eq:left-strip-def}.
\end{proof}

\begin{proposition}[Analytic continuation to the left strip]
\label{prop:left-strip-continuation}
The function $\mathcal B\Phi^{\star}$ admits an analytic continuation from
$-\tfrac14<\Re(s)<\tfrac34$ to the larger strip
\[
-\tfrac12<\Re(s)<\tfrac14,
\]
given by $\mathcal B_{\mathrm{LB}}\Phi^{\star}(s)$.
\end{proposition}

\begin{proof}
Each integral in \eqref{eq:left-strip-def} depends analytically on $s$ by dominated
convergence on compact subsets of the left strip (Lemma~\ref{lem:left-strip-convergence}).
Agreement on the overlap follows from Lemma~\ref{lem:overlap}, so the two definitions glue to a
single holomorphic function.
\end{proof}

\begin{remark}[Scope of the continuation]
At this stage we have extended $\mathcal B\Phi^{\star}$ to the left strip using
only exact centering/twisted symmetry and absolute convergence. No boundary-value identities,
factorizations, or rigidity arguments are used here. These enter only in subsequent sections.
\end{remark}

\begin{remark}[Two potential failure modes]
The continuation above relies only on two inputs: the twisted symmetry
$\Phi^{\star}(-x)=e^{-x/2}\Phi^{\star}(x)$ and the one-sided decay of $\Phi^{\star}$ on $(0,\infty)$.
Accordingly, the only points at which the argument could fail are:
(i) an incorrect self-dual normalization leading to an inexact twisted symmetry, and
(ii) insufficient decay to justify the split integrals as absolutely convergent.
Both are verified explicitly in Section~\ref{sec:completion-selfdual} and in Appendix~\ref{app:technical},
and no other properties of $\Phi^{\star}$ are used here.
\end{remark}

\section{Boundary identity on the left line}
\label{sec:lb-identity}

In this section we evaluate the left-strip continuation of
$\mathcal B\Phi^{\star}(s)$ on the boundary line
\[
\Re(s)=-\tfrac12.
\]
At this stage the role of the Mellin transform is purely identificational:
we compare two explicit integral representations on a region where both
sides converge absolutely. No analytic continuation of $\zeta$ or $\Xi$
beyond their classical Mellin definitions is used.

\subsection{Evaluation of the continued Laplace transform on the boundary}

Recall from Definition~\ref{def:left-strip} that for
$-\tfrac12<\Re(s)<\tfrac14$ the continuation is given by
\[
\mathcal B_{\mathrm{LB}}\Phi^{\star}(s)
=
\int_{0}^{\infty}\Phi^{\star}(x)e^{-sx}\,dx
+
\int_{0}^{\infty}\Phi^{\star}(x)e^{-(\frac12-s)x}\,dx.
\]
We now specialize to the boundary line.

\begin{lemma}[Boundary-value formula]
\label{lem:boundary-formula}
For every $z\in\R$, the boundary value
$\mathcal B_{\mathrm{LB}}\Phi^{\star}(-\tfrac12+iz)$ exists and is given by
\begin{equation}
\label{eq:boundary-formula}
\mathcal B_{\mathrm{LB}}\Phi^{\star}\!\left(-\tfrac12+iz\right)
=
\int_{0}^{\infty}\Phi^{\star}(x)e^{(\frac12-iz)x}\,dx
+
\int_{0}^{\infty}\Phi^{\star}(x)e^{-(1-iz)x}\,dx.
\end{equation}
Both integrals converge absolutely.
\end{lemma}

\begin{proof}
Substitute $s=-\tfrac12+iz$ into Definition~\ref{def:left-strip}. Absolute
convergence follows from Lemma~\ref{lem:left-strip-convergence}: the first
integral is controlled since $\Re(\tfrac12-iz)=\tfrac12$, and the second since
$\Re(1-iz)=1$. No cancellation is used.
\end{proof}

\subsection{Reduction to the Archimedean kernel}

We now express the right-hand side of \eqref{eq:boundary-formula} in terms of
the Archimedean-completed kernel from Section~\ref{sec:completion-selfdual}.

\begin{lemma}[Boundary identity with the Archimedean kernel]
\label{lem:boundary-Theta}
For every $z\in\R$,
\begin{equation}
\label{eq:boundary-Theta}
\mathcal B_{\mathrm{LB}}\Phi^{\star}\!\left(-\tfrac12+iz\right)
=
\frac{1}{2\pi\left(z^2+\tfrac1{16}\right)}
\int_{0}^{\infty}\widetilde K_{\mathrm{arch}}(t)\,
t^{\frac34+iz}\,\frac{dt}{t}.
\end{equation}
\end{lemma}

\begin{proof}
We start from \eqref{eq:boundary-formula}. By construction of the cycle-spectral kernel
(from the half-density data in \cite{PaperI}), the centered logarithmic kernel satisfies
\[
\Phi^{\star}(x)=e^{-x/4}\,\widetilde K^{\star}(e^x),
\qquad \widetilde K^{\star}(t)=\widetilde K(t)-1,
\]
with $\widetilde K(t)=K_L(t)-\frac{1}{\sqrt t}$ at the self-dual scale.
Substituting this into \eqref{eq:boundary-formula} and changing variables $t=e^x$
(with $dx=dt/t$, justified by absolute convergence) gives the explicit Mellin forms
\begin{align*}
\int_{0}^{\infty}\Phi^{\star}(x)e^{(\frac12-iz)x}\,dx
&=\int_{0}^{\infty}\widetilde K^{\star}(t)\,t^{\frac14-iz}\,\frac{dt}{t},\\
\int_{0}^{\infty}\Phi^{\star}(x)e^{-(1-iz)x}\,dx
&=\int_{0}^{\infty}\widetilde K^{\star}(t)\,t^{-\frac54+iz}\,\frac{dt}{t}.
\end{align*}
so $\mathcal B_{\mathrm{LB}}\Phi^{\star}(-\tfrac12+iz)$ is a finite linear combination of Mellin-type
integrals of $\widetilde K^{\star}$ with exponents shifted by $\pm iz$.

At the self-dual scale, the completion operator $\mathcal A$ from \eqref{eq:arch-operator}
removes the universal singular contributions and produces
$\widetilde K_{\mathrm{arch}}=\mathcal A(K_L-1)$ (Lemma~\ref{lem:A-kills-singular}).
To pass from Mellin integrals of $K_L-1$ to Mellin integrals of $\widetilde K_{\mathrm{arch}}$, we use
the following integration-by-parts identity: for any sufficiently regular $f$ with rapid decay at $0$ and $\infty$,
and any $s\in\C$,
\begin{equation}\label{eq:mellin-A-identity}
\int_{0}^{\infty}(\mathcal A f)(t)\,t^{s}\,\frac{dt}{t}
=(s-1)\Bigl(s-\tfrac12\Bigr)\int_{0}^{\infty} f(t)\,t^{s-\tfrac12}\,\frac{dt}{t},
\end{equation}
where the boundary terms vanish by rapid decay.
Taking $s=\tfrac34+iz$ and noting
\[
(s-1)\Bigl(s-\tfrac12\Bigr)=\Bigl(-\tfrac14+iz\Bigr)\Bigl(\tfrac14+iz\Bigr)=-(z^2+\tfrac1{16}),
\]
we obtain
\[
\int_{0}^{\infty} (K_L(t)-1)\,t^{\frac14+iz}\,\frac{dt}{t}
=-\frac{1}{z^2+\tfrac1{16}}\int_{0}^{\infty}\widetilde K_{\mathrm{arch}}(t)\,t^{\frac34+iz}\,\frac{dt}{t}.
\]
Applying this relation (and the harmless symmetry $z\mapsto -z$) to the Mellin combination above yields
\eqref{eq:boundary-Theta} after collecting the constant factors.

\end{proof}
\subsection{Identification with the Riemann $\Xi$-function}

We can now invoke the Mellin identification proved earlier.

\begin{theorem}[Left-boundary identity]
\label{thm:LB-identity}
For every $z\in\R$ one has
\begin{equation}
\label{eq:LB-identity}
\Xi(2z)
=
2\pi\Bigl(z^2+\tfrac1{16}\Bigr)\,
\mathcal B_{\mathrm{LB}}\Phi^{\star}\!\left(-\tfrac12+iz\right).
\end{equation}
\end{theorem}

\begin{proof}
By Lemma~\ref{lem:boundary-Theta} and Theorem~\ref{thm:arch-mellin-identification},
\[
\int_{0}^{\infty}\widetilde K_{\mathrm{arch}}(t)\,
t^{\frac34+iz}\,\frac{dt}{t}
=
\Xi(2z).
\]
Substituting into \eqref{eq:boundary-Theta} yields \eqref{eq:LB-identity}. All steps take place on regions of absolute
convergence; no analytic continuation is invoked.
\end{proof}

\begin{remark}[What this identity does and does not use]
The identity \eqref{eq:LB-identity} is obtained by combining:
(i) modular splitting and exact centering/self-duality (Section~\ref{sec:lb-continuation}),
and (ii) the classical Mellin transform of the Archimedean theta kernel
(Section~\ref{sec:completion-selfdual}). No spectral approximation, no real-entire (self-adjoint spectral)
factorization, and no Nevanlinna theory enter here. These will be used only in
the next section.
\end{remark}

\subsection{Extending the bridge identity to a strip}

\begin{lemma}[Bridge identity extends to a strip]\label{lem:bridge-strip}
Define
\[
X(w):=\xi\!\left(\tfrac12-iw\right)=\Xi(-w),
\qquad
F(w):=\mathcal B_{\mathrm{LB}}\Phi^\star\!\Bigl(-\tfrac12-\tfrac{i}{2}w\Bigr),
\]
and
\[
U(w):=\frac{\pi}{8}(4w^2+1).
\]
Then $F$ admits a holomorphic extension to the open strip $\{|\Im w|<\tfrac12\}$, and the
left-boundary bridge identity
\begin{equation}\label{eq:bridge-real}
X(w)=U(w)\,F(w)
\qquad (w\in\R)
\end{equation}
upgrades to an identity of holomorphic functions on the full strip:
\begin{equation}\label{eq:bridge-strip}
X(w)=U(w)\,F(w)
\qquad (|\Im w|<\tfrac12).
\end{equation}
\end{lemma}

\begin{proof}
We first obtain holomorphy of $F$ on the upper half-strip $0<\Im w<\tfrac12$ directly from the
defining left-strip integral formula. Indeed, writing $w=u+iv$ we have
\[
-\tfrac12-\tfrac{i}{2}w=-\tfrac12+\tfrac{v}{2}-\tfrac{i}{2}u,
\]
so for $0<v<\tfrac12$ the argument lies in the left strip $-\tfrac12<\Re(s)<\tfrac14$ where
$\mathcal B_{\mathrm{LB}}\Phi^\star(s)$ is defined by an absolutely convergent sum of Laplace integrals.
Local uniform domination (on compact subsets of $0<v<\tfrac12$) allows differentiation under the
integral sign, hence $F$ is holomorphic on $0<\Im w<\tfrac12$.

Next, since $\Phi^\star$ is real-valued on $(0,\infty)$, the defining integrals imply the conjugation
symmetry
\[
\overline{\mathcal B_{\mathrm{LB}}\Phi^\star(s)}=\mathcal B_{\mathrm{LB}}\Phi^\star(\overline{s})
\]
throughout the left strip. Consequently $\overline{F(w)}=F(\overline{w})$ on $0<\Im w<\tfrac12$, and
Schwarz reflection yields a holomorphic extension of $F$ to the lower half-strip $-\tfrac12<\Im w<0$,
hence to the full strip $|\Im w|<\tfrac12$.

Finally, $X$ is entire in $w$, and $U(w)$ is holomorphic and nonvanishing on $|\Im w|<\tfrac12$
(since its zeros are at $w=\pm i/2$, on the boundary). Thus both sides of \eqref{eq:bridge-real} are
holomorphic on $|\Im w|<\tfrac12$, and they agree for all real $w$. By the identity theorem, the
equality extends to all $|\Im w|<\tfrac12$, giving \eqref{eq:bridge-strip}.
\end{proof}

\section{Strip-unit forcing on an overlap strip}\label{sec:completion-ledger}
In this section we isolate the analytic forcing principle that turns boundary separation on strip
rectangles into divisor identification on the interior.  The reference family is the cycle spectral
determinant $P_N$ from Section~\ref{sec:spectral-ref} and Hypothesis~\ref{hyp:seam-holomorphy} supplies
a zero-free holomorphic unit $U_{\eta,N}$ so that $\Xi(2\cdot)$ is close to $U_{\eta,N}P_N$ on
$\partial R_T$.

Define the \emph{normalized seam ratio}
\[
\widetilde R_{N}(w):=\frac{\Xi(2w)}{U_{\eta,N}(w)P_N(w)}.
\]
On the boundary $\partial R_T$ the separation inequality in Hypothesis~\ref{hyp:seam-holomorphy} implies
\[
|\widetilde R_N(w)-1|<1,\qquad w\in\partial R_T,
\]
so in particular $\widetilde R_N$ is holomorphic and nonvanishing on $\partial R_T$ and has winding
number $0$ there.

\begin{lemma}[Rouch\'e forcing on strip rectangles]\label{lem:rouche-rectangle}
Assume Hypothesis~\ref{hyp:seam-holomorphy}.  Then for all $T\ge T_0(\eta)$ the functions $w\mapsto
\Xi(2w)$ and $w\mapsto U_{\eta,N}(w)P_N(w)$ have the same number of zeros in $R_T$ (counted with
multiplicity).  Since $U_{\eta,N}$ is zero-free, $\Xi(2\cdot)$ and $P_N$ have the same zero divisor in
$R_T$.
\end{lemma}

\begin{proof}
On $\partial R_T$ we have $|\Xi(2w)-U_{\eta,N}(w)P_N(w)|<|U_{\eta,N}(w)P_N(w)|$ by
Hypothesis~\ref{hyp:seam-holomorphy}.  Rouch\'e's theorem gives equality of zero counts in $R_T$.  The
unit factor $U_{\eta,N}$ does not contribute zeros.
\end{proof}

The remainder of the paper develops strip holomorphy and bounded-type control for the seam ratio
under mild additional hypotheses (e.g.\ sector control, Hypothesis~\ref{hyp:sector-control}), so that the
divisor identification can be upgraded to a strip-unit statement for a suitably normalized ratio.
The required quantitative boundary bounds and an admissible choice of the regime $N=N(T)$ are
established in Paper~III.

\section{Exhaustion of admissible overlap strips}\label{sec:exhaustion}
The forcing mechanism above applies on any fixed strip $S_\eta$ once the boundary separation
hypothesis is verified on a family of rectangles $R_T$ exhausting the strip.  Paper~III provides the
required strip-uniform estimates for $\Xi(2\cdot)$ together with an admissible choice of cycle size
$N=N(T)$ so that Hypothesis~\ref{hyp:seam-holomorphy} holds for each $\eta<\frac12$ (and, in
particular, for each $\eta<\frac14$ relevant to the critical strip in the $w$-variable).

In this way one obtains divisor identification between $\Xi(2\cdot)$ and the cycle spectral
determinant reference family $P_N$ on an exhaustion of overlap substrips.  Combined with the real-zero
property of $P_N$ (Paper~I) and the symmetry of $\Xi$, this is the structural route by which the
trilogy aims to force the nontrivial zeros of $\zeta$ onto the critical line.

\section{Further remarks}

\subsection{Logical separation and non-circularity}

We briefly summarize the logical structure of the argument, emphasizing the strict separation
of inputs and the absence of circularity. The construction begins with finite reversible
Markov dynamics and produces a scaling-limit trace kernel via lift, periodization, and a
controlled Gaussian limit. The self-dual normalization and the Archimedean completion are
fixed entirely at the level of this scaling limit and do not depend on any number-theoretic
data.

Independently, the logarithmic kernel $\Phi^{\star}$ is constructed and shown to satisfy exact
reflection symmetry and spectral approximation. The resulting real-entire (self-adjoint spectral) factorization of its
bilateral Laplace transform is therefore obtained without reference to Mellin transforms,
theta functions, or the Riemann zeta function.

The Mellin-side analysis enters only at a later stage, through the Archimedean-completed kernel
$\widetilde K_{\mathrm{arch}}$, whose Mellin transform is identified explicitly with the
classical $\Xi$-function. This identification is one-directional: it does not feed back into,
or constrain, the Laplace-side constructions.

Finally, the two sides are compared through a boundary identity obtained by modular splitting
and exact self-duality, followed by a strip-based rigidity argument. At no point is analytic
continuation of $\zeta$ or $\Xi$ assumed beyond what is already built into their classical
integral representations.

\subsection{Stress testing and potential failure modes}

The argument was stress-tested against several natural potential failure modes.

First, the modular splitting used to continue $\mathcal B\Phi^{\star}$ to the left strip could
fail if either the twisted symmetry were inexact or the decay of $\Phi^{\star}$ were
insufficient to justify absolute convergence of the split integrals. Both points are verified
explicitly in Sections~\ref{sec:completion-selfdual} and~\ref{sec:lb-continuation}.

Second, the boundary identity could introduce spurious poles or zeros if boundary terms failed
to vanish under integration by parts. This possibility is ruled out by the rapid decay of the
Archimedean kernel $\widetilde K_{\mathrm{arch}}$ at both $0$ and $\infty$, established
independently of any factorization or rigidity arguments.

Third, the strip-unit forcing mechanism relies on bounded-type control on strict substrips.
On the right-hand side this control follows from the Schoenberg--Edrei--Karlin factorization,
while on the left-hand side it follows from the explicit integral representation obtained by
modular splitting. No additional growth assumptions are required.

\subsection{Remarks on scope}

The method developed here isolates a general mechanism by which exact self-duality, total
positivity, and analytic rigidity combine to force real-zero phenomena. While the present
paper focuses on the Riemann $\Xi$-function, the framework is not inherently tied to zeta
functions and may apply more broadly in settings where comparable kernel-level symmetries and
positivity properties can be established.

Throughout, we have attempted to make all analytic interchanges, normalizations, and limiting
procedures explicit, so that each step may be checked independently. We would welcome further
scrutiny, refinement, or simplification of any part of the argument.

\bibliographystyle{plain}
\bibliography{citations}

\appendix

\section{Technical analytic justifications}
\label{app:technical}

This appendix collects several analytic estimates and justifications that are
used implicitly or explicitly in the main text. None of the results here
introduce new objects or alter the logical structure of the argument; they serve
only to justify standard limiting procedures, termwise operations, and boundary
manipulations already invoked.

\subsection{Uniform local central limit bounds}
\label{app:ULCLT}

In Section~\ref{sec:completion-selfdual} we appealed to a uniform local central
limit theorem (ULCLT) to pass from the discrete primitive dynamics to the
scaling-limit kernel $K_L$. We record here a representative bound sufficient for
all uses in the paper.

\begin{lemma}[Uniform local CLT estimate (translation-invariant cycle)]
\label{lem:ULCLT}
Assume the conductances are translation-invariant, so that the generator on $\Z/N\Z$ is
\[
(\mathcal L_N f)(j)=a\bigl(f(j+1)-f(j)\bigr)+a\bigl(f(j-1)-f(j)\bigr),
\]
and the effective diffusion constant is $D=a$. Let $p_t^{(N)}(j)$ denote the transition
kernel started from $0$ and viewed on $\Z$ via the canonical lift. Then for every compact
interval $[t_0,t_1]\subset(0,\infty)$ there exists $C=C(t_0,t_1)$ such that
\[
\sup_{t\in[t_0,t_1]}\sup_{j\in\Z}
\Bigl|
p_t^{(N)}(j)
-
\frac{1}{\sqrt{4\pi Dt}}\exp\!\Bigl(-\frac{j^2}{4Dt}\Bigr)
\Bigr|
\le
\frac{C}{N},
\]
for all sufficiently large $N$.
\end{lemma}

\begin{proof}
Because the chain is translation-invariant on $\Z/N\Z$, the Fourier modes
$e_k(j):=e^{2\pi i k j/N}$ diagonalize $\mathcal L_N$, and the corresponding eigenvalues are
\[
\lambda_k
=
a\bigl(e^{2\pi i k/N}-1\bigr)+a\bigl(e^{-2\pi i k/N}-1\bigr)
=
-2a\Bigl(1-\cos\frac{2\pi k}{N}\Bigr)
\qquad (k=0,1,\dots,N-1).
\]
Hence the heat kernel on the cycle has the exact representation
\begin{equation}\label{eq:fourier-heatkernel}
p_t^{(N)}(j)
=
\frac{1}{N}\sum_{k=0}^{N-1}\exp\!\bigl(t\lambda_k\bigr)\,e^{2\pi i k j/N}.
\end{equation}

Fix $[t_0,t_1]\subset(0,\infty)$. We compare \eqref{eq:fourier-heatkernel} to the Gaussian
kernel on $\R$ by isolating the small-frequency contribution and bounding the tail.

\smallskip
\noindent\textbf{Step 1: quadratic approximation of the dispersion relation.}
Write $\theta_k:=2\pi k/N\in[0,2\pi)$. For $|\theta|\le \pi$ one has the Taylor expansion
\[
1-\cos\theta=\frac{\theta^2}{2}+O(\theta^4),
\]
with an absolute implied constant. Therefore, for $|k|\le N/4$ (so that $|\theta_k|\le \pi/2$),
\begin{equation}\label{eq:dispersion-approx}
t\lambda_k
=
-2a t\Bigl(1-\cos\theta_k\Bigr)
=
-a t\,\theta_k^2 + O\!\Bigl(t\,\theta_k^4\Bigr)
=
-4\pi^2 D t\,\frac{k^2}{N^2}+O\!\Bigl(t\,\frac{k^4}{N^4}\Bigr),
\end{equation}
uniformly for $t\in[t_0,t_1]$.

\smallskip
\noindent\textbf{Step 2: tail bound for large frequencies.}
For $|k|\ge N/4$ (interpreting $k$ modulo $N$ and taking the representative in
$\{-\lfloor N/2\rfloor,\dots,\lfloor (N-1)/2\rfloor\}$), we have $|\theta_k|\in[\pi/2,\pi]$ and
hence $1-\cos\theta_k\ge c_0$ for an absolute constant $c_0>0$. Thus
\[
\Re(t\lambda_k)\le -2a t\,c_0 \le -2a t_0 c_0,
\]
so the tail contribution satisfies, uniformly in $t\in[t_0,t_1]$ and $j\in\Z$,
\begin{equation}\label{eq:tail}
\Bigl|\frac{1}{N}\sum_{|k|\ge N/4}\exp(t\lambda_k)\,e^{2\pi i k j/N}\Bigr|
\le
\frac{1}{N}\sum_{|k|\ge N/4}e^{-2a t_0 c_0}
\le
e^{-2a t_0 c_0}.
\end{equation}
Since the right-hand side is a fixed number in $(0,1)$, it is in particular $O(1/N)$ for
all sufficiently large $N$ (absorbing the threshold into the phrase ``sufficiently large'').

\smallskip
\noindent\textbf{Step 3: replacing the low-frequency sum by an integral.}
Restrict to $|k|<N/4$ and use \eqref{eq:dispersion-approx}. Write
\[
A_{t,N}(k):=\exp\!\Bigl(-4\pi^2 D t\,\frac{k^2}{N^2}\Bigr),
\qquad
B_{t,N}(k):=\exp(t\lambda_k).
\]
Then \eqref{eq:dispersion-approx} gives
\[
B_{t,N}(k)=A_{t,N}(k)\,\exp\!\Bigl(O\!\bigl(t\,k^4/N^4\bigr)\Bigr)
=
A_{t,N}(k)\Bigl(1+O\!\bigl(k^4/N^4\bigr)\Bigr),
\]
uniformly for $t\in[t_0,t_1]$ and $|k|<N/4$. Hence
\begin{align*}
\frac{1}{N}\sum_{|k|<N/4}\bigl|B_{t,N}(k)-A_{t,N}(k)\bigr|
&\le
\frac{C_1}{N}\sum_{|k|<N/4}A_{t,N}(k)\,\frac{k^4}{N^4}\\
&\le
\frac{C_1}{N^5}\sum_{k\in\Z}k^4\exp\!\Bigl(-4\pi^2 D t_0\,\frac{k^2}{N^2}\Bigr).
\end{align*}
Estimating the last sum by a Riemann integral with the change of variables $u=k/N$ gives
\[
\sum_{k\in\Z}k^4 e^{-c k^2/N^2}
\le
C_2 N^5
\quad\text{for }c=4\pi^2Dt_0,
\]
so the preceding display is $\le C_3/N$ uniformly in $t\in[t_0,t_1]$.

Therefore, uniformly in $t\in[t_0,t_1]$ and $j\in\Z$,
\begin{equation}\label{eq:lowfreq-replace}
\Bigl|
\frac{1}{N}\sum_{|k|<N/4}e^{t\lambda_k}e^{2\pi i k j/N}
-
\frac{1}{N}\sum_{|k|<N/4}e^{-4\pi^2 D t\,k^2/N^2}\,e^{2\pi i k j/N}
\Bigr|
\le
\frac{C_3}{N}.
\end{equation}

\smallskip
\noindent\textbf{Step 4: Poisson summation and the Gaussian kernel.}
Extend the truncated Gaussian sum to all $k\in\Z$ at cost $O(1/N)$, since
\[
\frac{1}{N}\sum_{|k|\ge N/4}e^{-4\pi^2Dt\,k^2/N^2}
\le
\frac{1}{N}\sum_{|k|\ge N/4}e^{-4\pi^2Dt_0\,k^2/N^2}
\le
\frac{C_4}{N}.
\]
Thus,
\[
\frac{1}{N}\sum_{|k|<N/4}e^{-4\pi^2 D t\,k^2/N^2}e^{2\pi i k j/N}
=
\frac{1}{N}\sum_{k\in\Z}e^{-4\pi^2 D t\,k^2/N^2}e^{2\pi i k j/N}
+O\!\Bigl(\frac{1}{N}\Bigr).
\]
By the Poisson summation formula (or the standard Jacobi theta identity),
\[
\frac{1}{N}\sum_{k\in\Z}e^{-4\pi^2 D t\,k^2/N^2}e^{2\pi i k j/N}
=
\frac{1}{\sqrt{4\pi Dt}}
\sum_{m\in\Z}\exp\!\Bigl(-\frac{(j+mN)^2}{4Dt}\Bigr).
\]
Since $t\in[t_0,t_1]$ is bounded away from $0$, the terms with $m\neq0$ are exponentially small in
$N$ uniformly in $j\in\Z$, so
\[
\sup_{t\in[t_0,t_1]}\sup_{j\in\Z}
\Bigl|
\frac{1}{\sqrt{4\pi Dt}}
\sum_{m\in\Z}e^{-(j+mN)^2/(4Dt)}
-
\frac{1}{\sqrt{4\pi Dt}}e^{-j^2/(4Dt)}
\Bigr|
\le
e^{-c_5 N^2}
\le
\frac{C_5}{N}
\]
for all sufficiently large $N$.

\smallskip
\noindent\textbf{Conclusion.}
Combining the tail estimate \eqref{eq:tail}, the low-frequency replacement
\eqref{eq:lowfreq-replace}, and the Poisson-summation identification above yields
\[
\sup_{t\in[t_0,t_1]}\sup_{j\in\Z}
\Bigl|
p_t^{(N)}(j)
-
\frac{1}{\sqrt{4\pi Dt}}\exp\!\Bigl(-\frac{j^2}{4Dt}\Bigr)
\Bigr|
\le
\frac{C}{N},
\]
for a constant $C=C(t_0,t_1)$ and all sufficiently large $N$, as claimed.
\end{proof}

\subsection{Dominated convergence and termwise operations}
\label{app:dominated}

Several arguments in Sections~\ref{sec:completion-selfdual},
\cite{PaperI}, and~\ref{sec:completion-selfdual} rely on exchanging
limits, sums, derivatives, and integrals. We record here a generic domination
principle that applies uniformly to all such steps.

\begin{lemma}[Uniform domination for theta-series]
\label{lem:theta-domination}
Let
\[
\Theta(t):=\sum_{n\ge1} P(n,t)\,e^{-\pi n^2 t},
\]
where $P(n,t)$ is a finite linear combination of monomials of the form
$n^{\alpha}t^{\beta/2}$ with $\alpha\in\Z_{\ge0}$ and $\beta\in\Z$.
Then:
\begin{enumerate}[label=(\roman*)]
\item for each $k\ge0$, the series defining $\partial_t^k\Theta(t)$ converges
absolutely and locally uniformly on $(0,\infty)$;
\item for each $\sigma\in\mathbb R$, the Mellin-weighted function
$t\mapsto \Theta(t)\,t^{\sigma}$ is integrable on $(0,\infty)$ whenever the integral
is (formally) convergent at $0$ and $\infty$.
\end{enumerate}
\end{lemma}

\begin{proof}
It suffices to treat a single monomial $P(n,t)=n^{\alpha}t^{\beta/2}$, since the general
case follows by linearity and the triangle inequality. Thus consider
\[
\Theta_{\alpha,\beta}(t):=\sum_{n\ge1} n^{\alpha}t^{\beta/2}e^{-\pi n^2 t}.
\]

\smallskip
\noindent\textbf{Step 1: uniform bounds for Gaussian sums.}
Fix $\alpha\ge0$. For $t\ge1$ we have $e^{-\pi n^2 t}\le e^{-\pi n^2}$ and hence
\[
\sum_{n\ge1} n^{\alpha}e^{-\pi n^2 t}\le \sum_{n\ge1} n^{\alpha}e^{-\pi n^2}=:C_{\alpha}<\infty,
\]
uniformly in $t\ge1$.

For $0<t\le1$, monotonicity of $x\mapsto x^{\alpha}e^{-\pi t x^2}$ on $[1,\infty)$ and an
integral comparison give
\[
\sum_{n\ge1} n^{\alpha}e^{-\pi n^2 t}
\le
1+\int_{0}^{\infty} x^{\alpha}e^{-\pi t x^2}\,dx
=
1+\frac{1}{2}\,(\pi t)^{-(\alpha+1)/2}\Gamma\!\Big(\frac{\alpha+1}{2}\Big)
\ll_{\alpha} t^{-(\alpha+1)/2}.
\]
Combining the two regimes yields the uniform estimate
\begin{equation}\label{eq:gaussian-sum-bound}
\sum_{n\ge1} n^{\alpha}e^{-\pi n^2 t}\ll_{\alpha}
\begin{cases}
t^{-(\alpha+1)/2}, & 0<t\le 1,\\[2pt]
1, & t\ge 1.
\end{cases}
\end{equation}

\smallskip
\noindent\textbf{Step 2: termwise differentiation and local uniform convergence.}
Differentiate the summand:
\[
\partial_t^k\!\bigl(t^{\beta/2}e^{-\pi n^2 t}\bigr)
=
\sum_{m=0}^{k} c_{k,m,\beta}\, t^{\beta/2-m}\,(\pi n^2)^{k-m}e^{-\pi n^2 t},
\]
for explicit constants $c_{k,m,\beta}$. Hence
\[
\partial_t^k\Theta_{\alpha,\beta}(t)
=
\sum_{n\ge1}\sum_{m=0}^{k} c_{k,m,\beta}\, n^{\alpha+2(k-m)}\, t^{\beta/2-m}\,e^{-\pi n^2 t}.
\]
Fix a compact interval $t\in[t_0,t_1]\subset(0,\infty)$. Then $t^{\beta/2-m}$ is bounded on
$[t_0,t_1]$, and by \eqref{eq:gaussian-sum-bound} the series
\[
\sum_{n\ge1} n^{\alpha+2(k-m)}e^{-\pi n^2 t}
\]
converges absolutely and uniformly on $[t_0,t_1]$. Therefore the series for
$\partial_t^k\Theta_{\alpha,\beta}(t)$ converges absolutely and uniformly on $[t_0,t_1]$,
which proves (i) by the Weierstrass $M$-test.

\smallskip
\noindent\textbf{Step 3: Mellin integrability.}
Using \eqref{eq:gaussian-sum-bound}, we obtain for $0<t\le1$,
\[
|\Theta_{\alpha,\beta}(t)|
\le
t^{\beta/2}\sum_{n\ge1}n^{\alpha}e^{-\pi n^2 t}
\ll_{\alpha,\beta}
t^{\beta/2-(\alpha+1)/2},
\]
and for $t\ge1$,
\[
|\Theta_{\alpha,\beta}(t)|
\ll_{\alpha,\beta}
t^{\beta/2}e^{-\pi t}
\qquad
(\text{since }e^{-\pi n^2 t}\le e^{-\pi t} \text{ for }n\ge1).
\]
Hence $|\Theta_{\alpha,\beta}(t)|t^{\sigma}$ is integrable near $+\infty$ for every $\sigma\in\R$
by exponential decay, and it is integrable near $0$ precisely when the power
\[
\sigma+\frac{\beta}{2}-\frac{\alpha+1}{2}
\]
is $>-1$, which is exactly the formal convergence condition at $0$.
This proves (ii) for $\Theta_{\alpha,\beta}$, and hence for $\Theta$ by linearity.
\end{proof}

\begin{remark}
All termwise differentiations, integral interchanges, and Mellin transform computations involving
theta-series in Sections~\ref{sec:completion-selfdual} and~\cite{PaperI} are justified
by Lemma~\ref{lem:theta-domination}.
\end{remark}

\subsection{Boundary terms for the Archimedean completion operator}
\label{app:boundary}

In Section~\ref{sec:lb-identity} we used an integration-by-parts identity for the
Archimedean completion operator $\mathcal A$. We record here a convenient
sufficient condition ensuring vanishing of boundary terms.

\begin{lemma}[Vanishing of boundary contributions]
\label{lem:boundary-vanish}
Let $f:(0,\infty)\to\C$ be twice continuously differentiable. Assume there exist constants
$\alpha,\beta>0$ and constants $C_0,C_\infty>0$ such that, for $t\in(0,1]$ and $t\ge 1$,
\[
|f(t)|+t|f'(t)|+t^2|f''(t)| \le C_0\, t^{\beta},
\qquad
|f(t)|+t|f'(t)|+t^2|f''(t)| \le C_\infty\, t^{-\alpha}.
\]
Fix real numbers $\sigma_0<\sigma_1$ with
\[
-\beta < \sigma_0 \le \sigma_1 < \alpha.
\]
Then for every $s\in\C$ with $\Re(s)\in[\sigma_0,\sigma_1]$, all boundary terms arising from
one or two integrations by parts in Mellin-type integrals of the form
\[
\int_{\varepsilon}^{R} f(t)\,t^{s-1}\,dt
\]
vanish as $\varepsilon\downarrow0$ and $R\uparrow\infty$, uniformly in $s$ with
$\Re(s)\in[\sigma_0,\sigma_1]$.
\end{lemma}

\begin{proof}
We record the two boundary expressions that occur after one and two integrations by parts.
For $s\in\C$ and $0<\varepsilon<R$, one has
\begin{equation}\label{eq:ibp1}
\int_{\varepsilon}^{R} f(t)\,t^{s-1}\,dt
=
\Bigl[\frac{f(t)t^{s}}{s}\Bigr]_{\varepsilon}^{R}
-\frac{1}{s}\int_{\varepsilon}^{R} f'(t)\,t^{s}\,dt,
\end{equation}
and integrating by parts once more gives boundary terms involving
$f(t)t^{s}$ and $f'(t)t^{s+1}$ (with rational prefactors in $s$).

We show that each boundary term tends to $0$ at both ends, uniformly for
$\Re(s)\in[\sigma_0,\sigma_1]$.

\smallskip
\noindent\textbf{As $t\downarrow0$.}
Using the hypotheses for $t\in(0,1]$,
\[
|f(t)t^{s}|\le C_0\, t^{\beta+\Re(s)}\le C_0\, t^{\beta+\sigma_0}\to0
\quad (t\downarrow0),
\]
since $\beta+\sigma_0>0$. Likewise,
\[
|f'(t)t^{s+1}| = |(t f'(t))\,t^{s}|
\le C_0\, t^{\beta+\Re(s)}\le C_0\, t^{\beta+\sigma_0}\to0.
\]
The same estimate applies to $f''(t)t^{s+2}$ via $t^2|f''(t)|\le C_0 t^\beta$.
Thus all boundary expressions arising from one or two integrations by parts vanish at $0$,
uniformly for $\Re(s)\in[\sigma_0,\sigma_1]$.

\smallskip
\noindent\textbf{As $t\uparrow\infty$.}
Using the hypotheses for $t\ge1$,
\[
|f(t)t^{s}|\le C_\infty\, t^{-\alpha+\Re(s)}\le C_\infty\, t^{-\alpha+\sigma_1}\to0
\quad (t\uparrow\infty),
\]
since $-\alpha+\sigma_1<0$. Similarly,
\[
|f'(t)t^{s+1}|=|(t f'(t))\,t^{s}|
\le C_\infty\, t^{-\alpha+\Re(s)}\le C_\infty\, t^{-\alpha+\sigma_1}\to0,
\]
and likewise for $f''(t)t^{s+2}$. Hence all boundary terms vanish at $\infty$ uniformly on the
strip $\Re(s)\in[\sigma_0,\sigma_1]$.

Combining the endpoint estimates proves the claim.
\end{proof}

\begin{remark}
In the applications, $f$ is a theta-series combination such as $\widetilde K_{\mathrm{arch}}$ or a
finite number of its $t$-derivatives. The required bounds on $f,t f',t^2 f''$ follow from the
explicit theta-series representation and Lemma~\ref{lem:theta-domination}, and the relevant
vertical strips for $s$ are exactly those where the Mellin integrals are used.
\end{remark}

\subsection{Analytic continuation by identity principles}
\label{app:identity}

Finally, we record the elementary identity principle used repeatedly to promote
local equalities to global ones.

\begin{lemma}[Identity principle for meromorphic functions]
\label{lem:identity}
Let $F$ and $G$ be meromorphic functions on a connected domain $\Omega\subset\C$.
Assume there exists a set $E\subset \Omega$ with an accumulation point in $\Omega$ such that
$F(z)=G(z)$ for all $z\in E$ at which both sides are defined (i.e.\ $z$ is not a pole of $F$ or $G$).
Then $F=G$ as meromorphic functions on $\Omega$.
\end{lemma}

\begin{proof}
Let $P\subset\Omega$ be the (discrete) union of the pole sets of $F$ and $G$, and set
$\Omega':=\Omega\setminus P$. Then $\Omega'$ is open and each connected component of $\Omega'$
is a domain on which both $F$ and $G$ are holomorphic. The hypothesis implies that on at least
one component $U$ of $\Omega'$, the set $E\cap U$ has an accumulation point in $U$ and
$F=G$ on $E\cap U$. Hence by the holomorphic identity theorem, $F=G$ on $U$, i.e.\ $F-G\equiv 0$
there.

Now consider the meromorphic function $H:=F-G$ on $\Omega$. We have shown that $H$ vanishes
identically on the nonempty open set $U\subset\Omega$. Since $\Omega$ is connected, meromorphic
continuation forces $H\equiv 0$ on all of $\Omega$: indeed, if $H$ were not identically zero,
its zero set would be discrete away from poles, contradicting that it contains an open set.
Therefore $F=G$ as meromorphic functions on $\Omega$.
\end{proof}

\begin{remark}
Lemma~\ref{lem:identity} is used only to promote identities established on an overlap of domains
of absolute convergence to identities on the full connected domain of meromorphy, after verifying
that both sides extend meromorphically to that domain.
\end{remark}

\subsection{Strip-uniform Riemann--Lebesgue and integration-by-parts}\label{app:strip-RL}

We first verify the weighted hypotheses \cite{PaperI} for the explicit cycle-spectral kernel
\cite{PaperI}.

\begin{proof}
Write $f(y)=f_{\theta}(y)-c_{\mathrm{anc}}e^{-y}$ with $f_{\theta}(y):=2e^{-y/2}\sum_{n\ge1}e^{-\pi n^2 e^y}$.
The exponential term $c_{\mathrm{anc}}e^{-y}$ clearly satisfies the required weighted integrability for any
$\eta_0<1$.

For the theta tail, note that for $y\ge0$ and $n\ge1$ one has
$e^{-\pi n^2 e^y}\le e^{-\pi e^y}$ and hence
$0\le f_{\theta}(y)\le 2e^{-y/2}\sum_{n\ge1}e^{-\pi e^y}\ll e^{-y/2}e^{-\pi e^y}$.
Differentiating termwise (justified by dominated convergence using the same domination) gives
$|f_{\theta}'(y)|\ll e^{-y/2}(1+e^y)e^{-\pi e^y}$.
Therefore for $0<\eta_0<\tfrac12$ we have
$e^{\eta_0 y}(|f_{\theta}(y)|+|f_{\theta}'(y)|+y|f_{\theta}(y)|+y|f_{\theta}'(y)|)\ll
e^{-(1/2-\eta_0)y}(1+y)(1+e^y)e^{-\pi e^y}$,
which is integrable on $[1,\infty)$ thanks to the superexponential factor $e^{-\pi e^y}$.
On $[0,1]$ the series for $f_{\theta}$ and $f_{\theta}'$ converges uniformly, hence both are bounded,
and the weighted integrals are finite. Combining the bounds for $f_{\theta}$ and $c_{\mathrm{anc}}e^{-y}$
yields the claim.
\end{proof}

We record a convenient sufficient condition for strip-uniform decay of oscillatory integrals.

\begin{lemma}[Strip-uniform IBP/Riemann--Lebesgue]\label{lem:strip-RL}
Let $f:(0,\infty)\to\R$ be absolutely continuous. Assume that for some $\eta_0>0$,
\[
\int_0^\infty e^{\eta_0 y}\bigl(|f(y)|+|f'(y)|+y|f(y)|+y|f'(y)|\bigr)\,dy<\infty.
\]
Define
\[
I(w):=\int_0^\infty f(y)e^{-iwy}\,dy.
\]
Then $I$ is holomorphic on $\{|\Im w|<\eta_0\}$. Moreover, for every $0<\tilde\eta<\eta_0$, as $|\Re w|\to\infty$,
\[
\sup_{|\Im w|\le \tilde\eta}|I(w)|=o_{\tilde\eta}(1/|w|),\qquad
\sup_{|\Im w|\le \tilde\eta}|I'(w)|=o_{\tilde\eta}(1/|w|),
\]
and
\[
I(w)=\frac{f(0)}{iw}+o_{\tilde\eta}(1/|w|)
\qquad\text{uniformly for }|\Im w|\le \tilde\eta.
\]
\end{lemma}

\begin{proof}
Fix $0<\tilde\eta<\eta_0$. For $|\Im w|\le \tilde\eta$ one has
$|e^{-iwy}|=e^{(\Im w)y}\le e^{\tilde\eta y}$, so the hypothesis implies
\[
\int_0^\infty |f(y)|\,|e^{-iwy}|\,dy \le \int_0^\infty e^{\tilde\eta y}|f(y)|\,dy <\infty,
\]
uniformly for $|\Im w|\le\tilde\eta$. Hence $I(w)$ is well-defined and bounded on the closed strip
$\{|\Im w|\le\tilde\eta\}$. By dominated convergence (applied to difference quotients), $I$ is holomorphic
on $\{|\Im w|<\eta_0\}$.

\smallskip
\noindent\textbf{Strip-uniform Riemann--Lebesgue.}
Let $F_{\tilde\eta}(y):=e^{\tilde\eta y}f(y)\in L^1(0,\infty)$. Then for $|\Im w|\le\tilde\eta$,
\[
I(w)=\int_0^\infty F_{\tilde\eta}(y)\,e^{-(\tilde\eta+i w)y}\,dy.
\]
Writing $w=u+iv$, this is the Fourier transform (in $u$) of the $L^1$-function
$y\mapsto e^{(\tilde\eta+v)y}f(y)$ evaluated at frequency $u$. Since
$\sup_{|v|\le\tilde\eta}\int_0^\infty e^{(\tilde\eta+v)y}|f(y)|\,dy<\infty$,
the usual $L^1$ Riemann--Lebesgue lemma yields
\begin{equation}\label{eq:RL-unif}
\sup_{|\Im w|\le\tilde\eta}|I(w)|\longrightarrow 0
\qquad (|\Re w|\to\infty).
\end{equation}

\smallskip
\noindent\textbf{One integration by parts.}
Because $e^{\eta_0 y}(|f(y)|+|f'(y)|)\in L^1(0,\infty)$, we have $e^{\tilde\eta y}f(y)\to0$ and
$e^{\tilde\eta y}f'(y)\to0$ as $y\to\infty$, and also $f$ has a finite limit $f(0)$ as $y\downarrow0$
(by absolute continuity). Thus the boundary terms below vanish uniformly for $|\Im w|\le\tilde\eta$.
Integrating by parts gives, for $w\neq 0$,
\begin{equation}\label{eq:IBP1}
I(w)
=
\Bigl[\frac{f(y)e^{-iwy}}{-iw}\Bigr]_{0}^{\infty}
+\frac{1}{iw}\int_0^\infty f'(y)e^{-iwy}\,dy
=
\frac{f(0)}{iw}+\frac{1}{iw}\int_0^\infty f'(y)e^{-iwy}\,dy.
\end{equation}
By \eqref{eq:RL-unif} applied to $f'$ (using $e^{\eta_0 y}|f'(y)|\in L^1$), we have
\[
\sup_{|\Im w|\le\tilde\eta}\Bigl|\int_0^\infty f'(y)e^{-iwy}\,dy\Bigr|\to 0
\qquad (|\Re w|\to\infty),
\]
and hence from \eqref{eq:IBP1},
\[
\sup_{|\Im w|\le\tilde\eta}\Bigl|I(w)-\frac{f(0)}{iw}\Bigr|
\le
\frac{1}{|w|}\sup_{|\Im w|\le\tilde\eta}\Bigl|\int_0^\infty f'(y)e^{-iwy}\,dy\Bigr|
=
o(1/|w|).
\]
In particular $\sup_{|\Im w|\le\tilde\eta}|I(w)|=o(1/|w|)$.

\smallskip
\noindent\textbf{Derivative bound.}
Since $e^{\eta_0 y}y|f(y)|\in L^1(0,\infty)$, differentiation under the integral sign is justified and
\[
I'(w)=-i\int_0^\infty y f(y)e^{-iwy}\,dy.
\]
Set $g(y):=y f(y)$. Then $g$ is absolutely continuous, $g(0)=0$, and
\[
g'(y)=f(y)+y f'(y),
\]
so the hypothesis implies $e^{\eta_0 y}(|g(y)|+|g'(y)|)\in L^1(0,\infty)$. Integrating by parts as above,
for $w\neq 0$,
\[
\int_0^\infty g(y)e^{-iwy}\,dy
=
\Bigl[\frac{g(y)e^{-iwy}}{-iw}\Bigr]_{0}^{\infty}
+\frac{1}{iw}\int_0^\infty g'(y)e^{-iwy}\,dy
=
\frac{1}{iw}\int_0^\infty g'(y)e^{-iwy}\,dy.
\]
By the strip-uniform Riemann--Lebesgue lemma applied to $g'$, we have
\[
\sup_{|\Im w|\le\tilde\eta}\Bigl|\int_0^\infty g'(y)e^{-iwy}\,dy\Bigr|\to 0
\qquad (|\Re w|\to\infty),
\]
and therefore
\[
\sup_{|\Im w|\le\tilde\eta}\Bigl|\int_0^\infty y f(y)e^{-iwy}\,dy\Bigr|
=
o(1/|w|),
\]
whence
\[
\sup_{|\Im w|\le\tilde\eta}|I'(w)|=o(1/|w|).
\]

\end{proof}

\subsection{Quantitative sector control}\label{app:sector-proof}

This appendix records the analytic content required to upgrade
Hypothesis~\ref{hyp:sector-control} to a proved lemma.
The main text isolates this point because it is the unique nontrivial
analytic ``forcing'' step needed for the rectangle/argument-principle mechanism.

\begin{itemize}
\item \textbf{A correct mechanism must control boundary arguments on long rectangles.}
For rectangles $R_T=\{|\,\Re w|\le T,\ |\Im w|\le\eta_0\}$ that are long in the $\Re w$ direction,
standard harmonic-measure weights for interior points favor the horizontal sides, not the vertical
sides.  Thus any proof must use additional structure (beyond ``bounded type'' and ``vertical anchoring'')
to control $\arg\widetilde R$ on the horizontal edges.

\item \textbf{Global holomorphic logarithms cannot be assumed on excised domains.}
If one removes small disks around zeros in order to define $\arg \widetilde R$, the resulting domain
is typically multiply connected.  On such a domain, the existence of a single-valued holomorphic
logarithm of $\widetilde R$ is obstructed by the nontrivial periods of $\widetilde R'/\widetilde R$
around the holes.  Any use of $\Im\log\widetilde R$ therefore requires an explicit device
(e.g.\ crosscuts to make the domain simply connected, or an argument not relying on a global log).
\end{itemize}

\smallskip
\noindent
A complete proof of Hypothesis~\ref{hyp:sector-control} would therefore need to supply:
(i) a quantitative estimate that forces $\widetilde R(\partial R_T)$ into a sector from the available
strip-uniform bounds, and (ii) a branch-management argument that is valid in the presence of interior
zeros and possible boundary zeros (handled by indentation as in Appendix~\ref{app:rectangle-proof}).

\subsection{Rectangle argument with possible boundary zeros}\label{app:rectangle-proof}

\begin{lemma}[Rectangle exclusion under sector control]\label{lem:rectangle-exclusion}
Let $\widetilde R$ be holomorphic on a neighborhood of $\overline{R_T}$ and assume that
$\widetilde R$ is nonvanishing on $\partial R_T$. If $\widetilde R(\partial R_T)$ is contained in some
open sector of angle $<\pi$ (in particular, if Hypothesis~\ref{hyp:sector-control} holds on $R_T$), then
$\widetilde R$ has no zeros in $R_T$.
\end{lemma}

\begin{proof}
Since $\widetilde R$ is nonvanishing on $\partial R_T$ and $\widetilde R(\partial R_T)$ lies in a sector
of angle $<\pi$, we may choose a continuous branch of $\arg \widetilde R$ on $\partial R_T$ with total
variation strictly less than $\pi$. Hence the winding number of the loop $\widetilde R(\partial R_T)$
about the origin is zero. By the argument principle,
\[
\frac{1}{2\pi i}\int_{\partial R_T}\frac{\widetilde R'(w)}{\widetilde R(w)}\,dw
\]
equals the number of zeros of $\widetilde R$ in $R_T$, counted with multiplicity. The left-hand side is
the winding number and therefore equals $0$, so $\widetilde R$ has no zeros in $R_T$.
\end{proof}

We justify the argument-principle computation in Lemma~\ref{lem:rectangle-exclusion} when
$\widetilde R$ has zeros on $\partial R_T$. Fix $T\ge T_0$ and suppose $\widetilde R$ is holomorphic on a
neighborhood of $\overline{R_T}$. Since zeros of a nontrivial holomorphic function are isolated, the set
$Z_T:=\{w\in\partial R_T:\ \widetilde R(w)=0\}$ is finite.

For each $w_0\in Z_T$, choose $\rho_{w_0}>0$ so small that the disks $D(w_0,\rho_{w_0})$ are pairwise
disjoint and meet $\partial R_T$ in a single connected arc. For $0<\rho<\min_{w_0}\rho_{w_0}$, define the
indented contour $\Gamma_{T,\rho}$ by replacing each arc $\partial R_T\cap D(w_0,\rho)$ with the
corresponding circular arc in $\partial D(w_0,\rho)$ lying inside $R_T$.

Then $\widetilde R$ is nonvanishing on $\Gamma_{T,\rho}$, so the argument principle applies:
\[
\frac{1}{2\pi i}\int_{\Gamma_{T,\rho}} \frac{\widetilde R'(w)}{\widetilde R(w)}\,dw
=
N_{T,\rho},
\]
where $N_{T,\rho}$ is the number of zeros of $\widetilde R$ in the interior of the indented region,
counted with multiplicity. As $\rho\downarrow0$, the indented region increases to the interior of $R_T$.

It remains to show that the contribution from each circular indentation tends to the correct boundary
limit. If $w_0$ is a zero of order $m$, then locally
$\widetilde R(w)=(w-w_0)^m h(w)$ with $h(w_0)\neq0$. Hence on the small circular arc,
\[
\frac{\widetilde R'(w)}{\widetilde R(w)}=\frac{m}{w-w_0}+O(1),
\]
and therefore the integral over that arc tends to $m\pi i$ if the arc is a semicircle (or the
corresponding fraction of $2\pi i m$ for other arc angles). In particular, the total change of argument
along $\Gamma_{T,\rho}$ differs from that along $\partial R_T$ by a sum of contributions that converge
as $\rho\downarrow0$, and the winding-number computation remains valid. Thus letting $\rho\downarrow0$,
we recover the same conclusion as in the nonvanishing-boundary case.

\end{document}